\documentclass[a4paper,12pt]{article}
\usepackage{amsmath,amsthm,amssymb,euscript,amscd}
\usepackage{graphicx}
\usepackage{xcolor}
\usepackage{enumerate}
\usepackage{tikz-cd}
\newtheorem{theorem}{Theorem}[section]

\newtheorem{lemma}[theorem]{Lemma}
\newtheorem{proposition}[theorem]{Proposition}
\theoremstyle{definition} 
\newtheorem{definition}[theorem]{Definition}

\providecommand{\keywords}[1]{\textbf{\textit{Keywords: }} #1}

\newcommand{\ah}{\ensuremath{{}^{\alpha}\! H}}

\DeclareMathOperator{\id}{id}

\DeclareMathOperator{\sgn}{sgn} 

\title{On some $H$-Galois objects which are distinguished by their polynomial $H$-identities}
\author{Waldeck Sch\"utzer\footnote{Department of Mathematics, UFSCar, Brazil, \texttt{waldeck@dm.ufscar.br}} \ and Abel Gomes de Oliveira\footnote{Department of Mathematics, UFSCar, Brazil, \texttt{abelgomes@dm.ufscar.br}, sponsored by CAPES.}}
\date{}
\begin{document}
\maketitle
 
\begin{abstract}
\noindent When $k$ is an algebraically closed field of characteristic 0 and $H$ is a non-semisimple monomial Hopf algebra, we show that all Galois objects over $H$ are determined up to $H$-comodule algebra isomorphism by their polynomial $H$-identities, extending a previous result in Kassel \cite{kassel}.
\end{abstract}

\section*{Introduction}

This paper contributes to the well-known question concerning whether the set of polynomials identities distinguishes PI-algebras (associative unital algebras over a field $k$ satisfying a nontrivial polynomial identity) up to isomorphism. For instance, it follows from the celebrated Amitsur-Levitsky theorem that the standard polynomial of degree $2n$ distinguishes the finite-dimensional central simple associative algebras over an algebraically closed field $k$ up to isomorphism. %
When $k$ is not algebraically closed, the situation can be quite different: the quaternions $\mathbb{H}$ are a central simple algebra of dimension $4$ over $\mathbb{R}$ and $\mathbb{C}\otimes_{\mathbb{R}}\mathbb{H}\cong\mathbb{C}\otimes_{\mathbb{R}}M_2(\mathbb{R})$, hence $\mathbb{H}$ and $M_2(\mathbb{R)}$ have the same set of polynomial identities, but they are obviously not isomorphic as algebras. 

When $k$ is algebraically closed and the $k$-algebras are ``simple'', various results have settled the isomorphism question in the affirmative (here the meaning of simple depends on the full structure of the algebra). For example, Kushkulei and Razmyslov \cite{kush} on simple Lie algebras, Drensky and Racine \cite{drensky} on simple Jordan algebras, Koshlukov and Zaicev \cite{koshlukov} on simple associative algebras graded by an abelian group, and Aljadeff and Haile \cite{aljadeffhaile} extending this result to any group. Also, Shestakov and Zaicev \cite{shestakov} on arbitrary finite dimensional simple algebras and, more recently, Bahturin and Yasumura \cite{felipe} on arbitrary (semigroup) simple graded algebras.

Extending the notion of graded polynomial identity, Kassel \cite{kassel} introduced a notion of polynomial $H$-identity, $H$ being an arbitrary Hopf algebra, and considered the $H$-comodule algebras, namely $k$-algebras which are also right $H$-comodules and the right coaction is compatible with the multiplication. When $H=kG$ is the group Hopf algebra of a group $G$, the $H$-comodule algebras are essentially the same as the $G$-graded associative algebras and the polynomial $H$-identities become the usual $G$-graded polynomial identities. More specifically, Kassel studied certain $H$-comodule algebras which are cleft extensions of the ground field $k$, namely the $H$-Galois objects, when $H$ is a (generalized) Taft algebra or the Hopf algebra $E(n)$, and showed that these objects are distinguished by their polynomial $H$-identities.

Following Kassel \cite{kassel}, here we consider the case when $H$ is a non-semisimple monomial Hopf algebra and prove the $H$-Galois objects are distinguished by their sets of polynomial $H$-identities.
%
%

The non-semisimple monomial Hopf algebras were classified by Chen et al. \cite{chen} and include the Taft algebras. This classification associates to each such an algebra the notion of ``group datum'' which will be described below. In turn, Bichon \cite{bichon} further classified the various group data into six different types and used this to determine, up to isomorphism, the $H$-Galois objects for the monomial Hopf algebras according to their associated group data. Our work draws from this classification.

As it turns out, the ``generalized'' Taft algebras are precisely the type I in Bichon's classification of the non-semisimple monomial Hopf algebras, and the isomorphism question for the corresponding $H$-Galois objects has already been settled in \cite{kassel}. Here we set out to investigate types II through VI.

Throughout this paper, $k$ denotes an algebraically closed field of characteristic 0.

\section{Preliminaries}

We shall follow closely the basic notations, definitions and results on Hopf algebras which are found for instance in \cite{montgomery} and \cite{dascalescu}. In particular, we assume that the reader is familiar with concepts such as Hopf algebras, $H$-comodule algebras, cleft extensions, and the Heyneman-Sweedler-type notation for the comultiplication $\Delta\colon H\rightarrow H\otimes H$ and the coaction $\delta \colon M\rightarrow M\otimes H$ for a (right) $H$-comodule $M$. As usual, the counit is $\varepsilon \colon  H\rightarrow k$. Concerning the polynomial $H$-identities we keep the same notations and definitions from \cite{kassel}. These will be briefly recalled in this section.

\subsection{$H$-comodule algebras and $H$-Galois objects}\label{galois sec}
We recall the definition of the twisted Hopf algebra $\ah$.  If $H$ is a Hopf algebra and
$\alpha\in Z^2(H,k^{\times})$
 is a normalized convolution invertible (right) 2-cocycle then the algebra $\ah$ is defined as an algebra $\mathcal{V}_H$ with the same vector space structure of $H$ and whose multiplication given by
\[
v_xv_y=\alpha(x_1,y_1)v_{x_2 y_2}
\] 
where the $v$-symbols are the images of the elements of $H$ under the linear isomorphism $x \mapsto v_x$ onto $\mathcal{V}_H$, and $x_1\otimes x_2$ is the Heyneman-Sweedler notation for $\Delta (x)=\sum_{i}x_{i1}\otimes x_{i2}$.
The cocycle condition
\[
 \alpha(x_1,\,y_1)\alpha(x_2y_2,\,z) = \alpha(y_1,\,z_1)\alpha(x,\,y_2z_2),
\]
for all $x,y,z\in H$, is responsible for the associativity of the multiplication in $\ah$, and the normalization
\[
\alpha(x,1) = \alpha(1,x)=\varepsilon(x),
\]
for all $x\in H$, implies that $1_{\ah}=u_1$. Moreover, $\ah$ is an $H$-comodule algebra with coaction $\delta \colon \ah \longrightarrow \ah \otimes H$ given by
\[
\delta(v_x)=v_{x_1}\otimes x_2.
\]  

Next we recall that an $H$-comodule algebra $A$ (with coaction $\rho \colon A \longrightarrow A \otimes H$) is called a (right) $H$-Galois object if the subalgebra of coinvariants $A^{coH}=\{ a \in A \mid \rho (a)=a \otimes 1_H \}$ is isomorphic to $k$ and the map 
\[
\begin{array}{rccc}
\beta \colon&  A\otimes A &\longrightarrow & A \otimes H \\
  &a\otimes b &\longmapsto & (a\otimes 1_H)\rho(b)
\end{array}
\]
is a linear isomorphism.

The next two propositions show that the $H$-Galois objects over finite dimensional Hopf algebras are equivalent to the twisted $H$-comodule algebras $\ah$ for some 2-cocycle $\alpha$.  The proofs are based on the discussions found in chapters 7 and 8 of \cite{montgomery}.

\begin{proposition}Let $\alpha$ be a normalized convolution invertible 2-cocycle. Then $\ah$ is an $H$-Galois object.
\end{proposition}
\begin{proof}In fact, $x\in H^{coH}$ if, and only if, $\Delta (x)=x\otimes 1_H$,
hence $$x=xS(1_H)=x_1 S(x_2)=\varepsilon(x)1_H.$$ It follows that $u_x\in (\ah)^{coH}$ if, and only if, $u_x=u_{\varepsilon(x)1_H}=\varepsilon(x)u_{1_H}$, which shows that $(\ah)^{coH}\cong k$.

Let $\varphi\colon  H\rightarrow \ah$ be the linear map given by
\[
\varphi(x) = \sigma^{-1}(S(x_2),x_3)v_{S(x_1)}.
\]
It is straightforward to check that $\varphi$ is the convolution inverse to the isomorphism $x\mapsto v_x$, that is $v_{x_1}\varphi(x_2)=\varepsilon(x)v_{1_H}$.
Using this map, define $\gamma \colon \ah\otimes H\rightarrow \ah\otimes \ah$ by
\[
\gamma(v_x\otimes y)=v_x\varphi(y_1)\otimes v_{y_2}.
\]
A direct calculation shows that $\beta$ and $\gamma$ are mutual inverses:
\begin{align*}
\beta\gamma(v_x\otimes y) &= \beta(v_x\varphi(y_1)\otimes y_2) \\
 &= v_x\varphi(y_1)v_{y_2}\otimes y_3 \\
 &= v_x\varepsilon(y_1)v_{1_H}\otimes y_2 = v_x\otimes y,
\end{align*}
and
\begin{align*}
\gamma\beta(v_x\otimes v_y) &= \gamma(v_xv_{y_1}\otimes y_2) \\
 &= v_x u_{y_1}\varphi(y_2)\otimes v_{y_3} \\
 &= v_x \varepsilon(y_1)v_{1_H}\otimes v_{y_2} = v_x\otimes v_y,
\end{align*}
for all $x,y\in H$. This completes the proof.
\end{proof}

\begin{proposition}\label{prop:H-existence}Let $H$ be a finite dimensional Hopf algebra and $A$ an $H$-Galois object. Then there is a normalized convolution invertible 2-cocycle $\alpha$ such that $A$ and $\ah$ are isomorphic as $H$-comodule algebras. 
\end{proposition}
\begin{proof}In \cite[Proposition 2]{kreimer} it is proved that $A$ is isomorphic to $H^{*}$ as $H^{*}$-modules. In particular $\dim_{k} A$ is finite, hence from \cite[Corollary 3.1.6]{radford} $A$ is a rational $H^{*}$-module. It follows from \cite[Proposition 3.2.2(b)]{radford} that $H$ is isomorphic to $A$ as $H$-comodules. Let $\phi\colon H\rightarrow A$ be this isomorphism. From the proof in \cite[Theorem 8.2.4]{montgomery}, $\phi$ has a convolution inverse $\phi^{-1}\colon H\rightarrow A$. As a consequence,
we can always assume that $\phi(1_H)=1_A$, for otherwhise we can simply
replace $\phi$ with $\phi'=\phi(1_H)^{-1}\phi$. 

Using these maps, we define
\[
\alpha(x,y) = \phi(x_1)\phi(y_1)\phi^{-1}(x_2y_2).
\] 
To see that $\alpha$ is a cocycle, we first check that it's image lies in $k{1_A}$. Since $A^{coH}\cong k$, it is enough to show that $\delta(\alpha(x,y))=\alpha(x,y)\otimes 1_H$ ($\delta$ being the coaction as in  \cite[Corollary 3.1.6]{radford}). Following the proof of \cite[Proposition 7.2.3]{montgomery}, observing that $\delta\phi = (\phi\otimes \id_H)\Delta$ (since $\phi$ is $H$-comodule map), for all $x,y\in H$, we calculate
\begin{align*}
\delta(\alpha(x,y)) &= \delta\phi(x_1)\delta\phi(y_1)\delta\phi^{-1}(x_2y_2) \\
&= \left(\phi(x_{11})\otimes x_{12} \right)\left(\phi(y_{11})\otimes y_{12}\right)\left(\phi^{-1}(x_{22}y_{22})\otimes S(x_{21}y_{21}) \right)\\
&= \phi(x_1)\phi(y_1)\phi^{-1}(x_4y_4)\otimes x_2y_2S(y_3)S(x_3) \\
&= \phi(x_1)\phi(y_1)\phi^{-1}(x_2y_2)\otimes 1_H = \alpha(x,y)\otimes 1_H,
\end{align*}
as required. Next, we have
\[
\alpha(x,1_H)=\phi(x_1)\phi(1_H)\phi^{-1}(x_21_H)=\phi(x_1)\phi^{-1}(x_2)=	\varepsilon(x)=\alpha(1_H,x),
\]
for all $\alpha \in H$, so $\alpha$ is normalized. It remains to check the cocycle condition:
\begin{align*}
\alpha(x_1,y_1)\alpha(x_2y_2,z) &= \phi(x_1)\phi(y_1)\varepsilon(x_2y_2)\phi(z_1)\phi^{-1}(x_3y_3z_2) \\
&= \phi(x_1)\phi(y_1)\phi(z_1)\phi^{-1}(x_2y_2z_2) \\
&= \phi(y_1)\phi(z_1)\phi(x_1)\varepsilon(y_2z_2)\phi^{-1}(x_2y_3z_3) \\
&= \alpha(y_1,z_1)\alpha(x,y_2z_2),
\end{align*}
for $x,y,z\in H$.

To complete the proof, we consider the algebra $\ah$ and the map
$F:\ah \rightarrow A$ given by $F(v_x)=\phi(x)$. It is clear that $F$ is
an $H$-comodule isomorphism, so it is enough to check that $F$ is an algebra
map:
\begin{align*}
 F(v_x v_y) &= \alpha(x_1,y_1) \phi(x_2y_2) \\
  &= \phi(x_1)\phi(y_1)\phi^{-1}(x_2y_2)\phi(x_3y_3) \\
  &= \phi(x_1)\phi(y_1)\varepsilon(x_2)\varepsilon(y_2) \\
  &= F(v_x)F(v_y),
\end{align*}
and
\[
   F(v_{1_H}) = \phi(1_H) = 1_A,
\]
as required. Therefore, $A$ is isomorphic to $\ah$ as $H$-comodule algebras.
\end{proof}

\subsection{Polynomial $H$-identities}
Following \cite{kassel}, for each $i=1,2,\ldots$ let $X_i^H$ be a copy of a Hopf algebra $H$ and denote by $X_i^x \ (x \in H)$ the elements of $X_i^H$ (called $X$-symbols). Define $X_H=\bigoplus_{i\geqslant 1} X_i^H$ and take the tensor algebra
\[
T=T(X_H)=T\left( \displaystyle \bigoplus_{i\geqslant 1} X_i^H \right).
\]
$T$ is an $H$-comodule algebra with coaction given by
\[
\delta(X_i^x)=X_i^{x_1}\otimes x_2.
\]

We shall need a symmetric version $S$ of $T$ which is defined naturally as the symmetric algebra of $X_H$. However, to avoid confusion, when referring to $S$ we replace the $X$-symbols $X_i^x$ by the $t$-symbols $t_i^x$.
 
Given a linear basis $\{x_1, \ldots, x_r\}$ for $H$, it is easy to see that the tensor algebra $T$ is isomorphic to free associative unital algebra defined by the indeterminates $\bigcup_{i\geq 1}\{X_i^{x_j}\,|\,1\leq j\leq r\}$, while the symmetric algebra $S$ is isomorphic to the algebra of commutative polynomials in the indeterminates $\bigcup_{i\geq 1}\{t_i^{x_j}\,|\,1\leq j\leq r\}$. In view of this
remark it should be clear that the following definition generalizes both the ordinary (case $H=k$) and the $G$-graded (case $H=kG$) polynomial identities:
\begin{definition}\cite[Definition 2.1]{kassel}
An element $P \in T$ is a polynomial $H$-identity for the $H$-comodule algebra $A$ if $\mu(P)=0$ for all $H$-comodule algebra maps $\mu \colon T \longrightarrow A$.
\end{definition}


We denote the set of the polynomial $H$-identities for an algebra $A$ by $I_H(A)$. Then
\[
I_H(A)=\displaystyle\bigcap_{\mu\colon  T \longrightarrow A} \ker \mu.
\]
Central in PI-theory is the study of T-ideals, which can be very difficult. For $A=\ah$, it happens that $I_H(A)$ can be characterized as the kernel of a single $H$-comodule algebra map, to be described bellow.

Consider the algebra $S\otimes \ah$
generated by the simple tensors $t_i^{x}\otimes v_y$ ($x,y\in H$), hereby
denoted simply by $t_i^{x}v_y$. With the coaction
\[
\delta (t_i^xv_y)=t_i^xv_{y_1}\otimes y_2,
\]
$S\otimes \ah$ becomes an $H$-comodule algebra. Also consider the map
\[
\mu_{\alpha} \colon T \longrightarrow S \otimes \ah,
\]
defined by $\mu_{\alpha}(X_i^x)=\sum t_i^{x_1}v_{x_2}$. A direct
calculation shows that this is an $H$-comodule algebra map, and in \cite[Proposition 2.7]{kassel} it is shown that every $H$-comodule
algebra map $\mu\colon T\rightarrow \ah$ factors through $\mu_{\alpha}$,
in the sense that given $\mu$ there is a unique algebra map $\xi\colon S\rightarrow k$ such that
\[
   \mu = (\xi\otimes \id_H)\circ \mu_{\alpha}.
\]
The usefulness of $\mu_{\alpha}$ becomes clear in the following result.

\begin{theorem}\cite[Theorem 2.6]{kassel} \label{mu} An element $P \in T$ is a polynomial $H$-identity for $\ah$ if and only if $\mu_{\alpha}(P)=0$. Equivalently, $I_H(\ah)=\ker \mu_{\alpha}$.
\end{theorem}

\section{Non-semisimple Monomial Hopf Algebras and their Galois objects}
In this section, we recall the classification \cite{chen} of the non-semisimple monomial Hopf Algebras and the classification \cite{bichon} of their Galois objects.

Following Chen, Huang, Ye and Zhang \cite{chen} we define the group datum:

\begin{definition}
A group datum is a quadruplet $\mathbb{G}=(G,g,\chi, \mu)$, where:
\begin{enumerate}[(i)]
	\item $G$ is a finite group, with an element $g$ in its center;
	\item $\chi \colon G \longrightarrow k^{\times}$ is a one-dimensional representation with $\chi(g) \neq 1$;
	\item $\mu \in k$ is such that $\mu=0$ if $o(g)=o(\chi(g))$ and, if $\mu \neq 0$, then $\chi^{o(\chi(g))}=1$.  
\end{enumerate}
\end{definition}

For each group datum $\mathbb{G}=(G,g,\chi, \mu)$, they associate an associative unital algebra $A(\mathbb{G})$ with generators $x\in G$ and $y$ satisfying the relations
\[
yx=\chi(x)xy \text{ and }  y^d=\mu(1-g^d),
\]
for all $x \in G$, where $d=o(\chi(g))$. It follows from the Diamond's Lemma, that $\{xy^i\,|\,x\in G,\,0\leq i\leq d-1\}$ forms a basis for $A(\mathbb{G})$, therefore its dimension is $|G|d$. One can endow this algebra with a Hopf algebra structure with comultiplication $\Delta$, counit $\varepsilon$ and antipode $S$ by
\begin{align*}
\Delta(y)=1\otimes y +y \otimes g, \ \varepsilon(y)=0, \ S(y)=-yg^{-1}, \\
\Delta(x)=x\otimes x, \ \varepsilon(x)=1, \ S(x)=x^{-1}, \ \forall x \in G.
\end{align*}
Next, they show that this class of Hopf algebras is precisely that of the non-semismiple monomial Hopf algebras (the precise definition of a monomial Hopf algebra is given in \cite{chen}). Observe that $A(\mathbb{G})$ is indeed non-semisimple, since it has finite dimension and $S^2\neq 1$.

For the $A(\mathbb{G})$-Galois objects, following Bichon \cite{bichon}, we define $n=o(g)$, $d=o(\chi(g))$, and $q=\chi(g)$ a primitive $d$-th root of the unity, and further divide the various group data $\mathbb{G}=(G,g,\chi,\mu)$ into 6 different types:
\begin{itemize}
\item Type I: $\mu=0$, $d=n$ and $\chi^d=1$;
\item Type II: $\mu=0$, $d=n$ and $\chi^d\neq 1$;
\item Type III: $\mu=0$, $d<n$ and $\chi^d= 1$;
\item Type IV: $\mu=0$, $d<n$, $\chi^d\neq 1$ and no $\sigma \in Z^2(G,k^{\times})$, with $\sigma(g^d,x)=\chi^d(x)\sigma(x,g^d)$, exists for all $x \in G$;
\item Type V: $\mu=0$, $d<n$, $\chi^d\neq 1$ and there exists a $\sigma \in Z^2(G,k^{\times})$ with $\sigma(g^d,x)=\chi^d(x)\sigma(x,g^d)$, for all $x \in G$;
\item Type VI: $\mu\neq 0$ (and hence $d<n$ and $\chi^d=1$).
\end{itemize}

Here we observe that the Taft algebra $H_{n^2}$ occurs as a particular case of the type I $A(\mathbb{G})$ when $g$ is a generator of $G=\mathbb{Z}/n\mathbb{Z}$.

To each $A(\mathbb{G})$, 2-cocycle $\sigma$, and $a\in k$, Bichon constructs an associative unital algebra $A_{\sigma,a}(\mathbb{G})$ as follows:

\begin{definition}\cite[Definition 2.2]{bichon}
Let $\sigma \in Z^2(G, k^{\times})$ and $a \in k$. We define the algebra $A_{\sigma,a}(\mathbb{G)}$ to be the algebra with generators $u_y$, $\{u_x\}_{x \in G}$ and subject to the relations:
\[
u_{x}u_{x'}=\sigma(x,x')u_{xx'}, \ u_1=1, \ u_yu_x=\chi(x)u_xu_y \text{ and } u_y^d=au_{g^d},
\]
for all $x,x' \in G$.
\end{definition}
The following proposition shows that $A_{\sigma,a}(\mathbb{G})$ is always an $A(\mathbb{G})$-comodule algebra, and gives a necessary and sufficient condition for this algebra to be an $A(\mathbb{G})$-Galois object. 

\begin{proposition}\cite[Proposition 2.3]{bichon}\label{prop:condition galois} The algebra $A_{\sigma, a}(\mathbb{G})$ has a right $A(\mathbb{G})$-comodule algebra structure with coaction $\rho \colon A_{\sigma, a}(\mathbb{G}) \longrightarrow A_{\sigma, a}(\mathbb{G}) \otimes A(\mathbb{G})$ defined by $\rho(u_y)=u_1\otimes y +u_y \otimes g$ and $\rho(u_x)=u_x\otimes x$ for all $x \in G$. Moreover, $A_{\sigma, a}(\mathbb{G})$ is an $A(\mathbb{G})$-Galois object  if and only if
	\begin{equation}\label{eq iso}
		a \sigma(g^d,x)=a\chi(x)^d\sigma(x,g^d),
	\end{equation}
	for all $x\in G$. In this case, the set $\{u_xu_y^i \mid x \in G \text{ and } 0 \leqslant i \leqslant d-1 \}$ is a linear basis of $A_{\sigma, a}(\mathbb{G})$ and the map $\Psi \colon A(\mathbb{G}) \longrightarrow A_{\sigma, a}(\mathbb{G})$, $xy^i \mapsto u_xu_y^i$ is an isomorphism of $A(\mathbb{G})$-comodules.
\end{proposition}

Next, Bichon shows that all $A(\mathbb{G})$-Galois objects are of the form $A_{\sigma, a}(\mathbb{G})$ and gives a necessary and sufficient condition for two such objects to be isomorphic as $A(\mathbb{G})$-comodule algebras.

\begin{proposition}\cite[Proposition 2.9]{bichon}\label{galois iso} Let $B$ be an $A(\mathbb{G})$-Galois object. Then there exists $\sigma \in Z^2(G,k^{\times})$ and $a \in k$ such that $B\cong A_{\sigma, a}(\mathbb{G})$ as $A(\mathbb{G})$-comodule algebras.
\end{proposition}

\begin{proposition}\cite[Proposition 2.10]{bichon}\label{iso condition} Let $\sigma, \tau \in Z^2(G,k^{\times})$ and $a,b \in k$ such that $A_{\sigma, a}(\mathbb{G})$ and $A_{\tau, b}(\mathbb{G})$ are $A(\mathbb{G})$-Galois objects. Then the $A(\mathbb{G})$-comodule algebras $A_{\sigma, a}(\mathbb{G})$ and $A_{\tau, b}(\mathbb{G})$ are isomorphic if and only if exists $\nu \colon G \longrightarrow k^{\times}$ with $\nu(1)=1$ such that
	\[
	\sigma=\partial(\nu)\tau \text{ and } b=a \nu(g^d).
	\]
\end{proposition}

It turns out that here we shall not need all of the details of Bichon's classification, but only a few facts that we have tailored to our needs in the following propositions.

\begin{proposition}\label{prop:iso2} Let $\mathbb{G}$ be a group datum of types II or IV. Then any $A(\mathbb{G})$-Galois object is isomorphic to $A_{\sigma, 0}(\mathbb{G})$ for some $\sigma \in Z^2(G,k^{\times})$.
\end{proposition}
\begin{proof}
	Let $B$ an $A(\mathbb{G})$-Galois object. By Proposition \ref{galois iso}, $B\cong A_{\sigma,a}(\mathbb{G})$ for some $\sigma \in Z^2(G,k^{\times})$ and $a \in k$. 
	
	In type II, $d=n$ and $\chi^d \neq 1$, while in type IV there is no $\sigma \in Z^2(G,k^{\times})$ satisfying $\sigma(g^d,x)=\chi^d(x)\sigma(x,g^d)$, for all $x \in G$.
In both cases, condition (\ref{eq iso}) is satisfied only if $a=0$, hence $B\cong A_{\sigma,0}$.
\end{proof}

\begin{proposition} \label{prop:iso3}Let $\mathbb{G}$ be a group datum of type III, V or VI. Then	 any $A(\mathbb{G})$-Galois object is isomorphic to either $A_{\sigma, 0}(\mathbb{G})$ or $A_{\sigma, 1}(\mathbb{G})$ for some $\sigma \in Z^2(G,k^{\times})$.
\end{proposition}
\begin{proof}
	Let $B$ be an arbitrary $A(\mathbb{G})$-Galois object. By Proposition \ref{galois iso}, $B\cong A_{\tau,a}(\mathbb{G})$ for some $\tau \in Z^2(G,k^{\times})$, $a \in k$ and we may assume that $a \neq 0$ ($A_{\tau,0}$ is always a Galois object).
	 
	Let $\nu \colon G \longrightarrow k^{\times}$ be such that $\nu(1)=1$ and $\nu(g^d)=a$ (notice that $g^d\neq 1$). For $\sigma=\partial(\nu)\tau\in Z^2(G,k^\times)$, 
	\[
	\sigma(g^d,x)=\partial(\nu)\tau(g^d,x)=\partial(\nu)\chi(x)^d\tau(x,g^d)=\chi(x)^d\sigma(x,g^d),
	\]
	for all $x \in G$, hence from Proposition \ref{prop:condition galois}, $A_{\sigma, 1}(\mathbb{G})$ is an $A(\mathbb{G})$-Galois object. It follows from Proposition \ref{iso condition} that $A_{\tau, a}(\mathbb{G})\cong A_{\sigma,1}(\mathbb{G})$.
\end{proof}

\section{The Main Theorem}

In this section, $H$ always denotes the non-semisimple monomial Hopf algebra $A(\mathbb{G})$, for some group datum $\mathbb{G}$, and our main objective is to prove the following result.

\begin{theorem}[Main]\label{thm:main}
Suppose that $A$ and $B$ are two $H$-Galois objects. If 
\[
I_H(A)=I_H(B),
\]
then $A$ and $B$ are isomorphic as $H$-comodule algebras.
\end{theorem}

The first step is to check that the $H$-Galois objects $A_{\sigma,a}$ are indeed polynomial $H$-identity algebras. Here it is useful the fact that every $H$-Galois object is isomorphic to $\ah$ for some normalized convolution invertible 2-cocycle $\alpha$ (see Proposition \ref{prop:H-existence}). It is easy to check that this isomorphism is the $H$-comodule algebra map $F:A_{\sigma,a}\rightarrow \ah$ given by $F(u_xu_y^i)=v_{xy^i}$.  In particular, since 
\[
v_{x}v_{x'}=\sigma(x,x')v_{xx'}, \ v_1=1, \ v_yv_x=\chi(x)v_xv_y \text{ and } v_y^d=av_{g^d},
\]
for all $x,x' \in G$,
any valid relation in $A_{\sigma, a}(\mathbb{G})$ can be transferred to a same-format relation in $\ah$.

To simplify the notation we set:
\[
\begin{aligned}[c]
X&=X_1^g, \\
Y&=X_1^y, \\
E&=X_1^1, 
\end{aligned}
\qquad
\begin{aligned}[c]
t_g&=t_1^g, \\
t_y&=t_1^y, \\
t_1&=t_1^1.
\end{aligned} 
\]
For use below, we first record an easy lemma.
\begin{lemma}\label{qcommute}
	Fix $\zeta$ a primitive $m$-th root of unit and let $z$ and $w$ two variables such that  $zw=\zeta wz$. Then
	\[
	(z+w)^m=z^m+w^m.
	\]
\end{lemma}
\begin{proof}
From \cite[Proposition IV.2.2]{kasselbook}, we have the $\zeta$-analogue of the Binomial Theorem,
\[
(z+w)^m = \sum_{i=0}^m {m \choose i}_{\zeta} z^{i}w^{m-i},
\]
where
\[
{m\choose i}_{\zeta} = \frac{(\zeta^m-1)\cdots(\zeta^{m-i+1}-1)}{(\zeta^i-1)\cdots(\zeta-1)}
\]
is the $\zeta$-binomial. Obviously, ${m\choose 0}_\zeta = {m \choose m}_\zeta=1$ and, whenever $0 < i < m$, the factor $\zeta^m-1=0$ remains in the numerator of ${m\choose i}_\zeta$ while no factor of the denominator is null. Therefore, ${m\choose i}_\zeta=0$ in this case.
\end{proof}

\begin{proposition}\label{prop:H-PI}
	Let $H$ be a non-semisimple monomial Hopf algebra and $A_{\sigma, a}{(\mathbb{G})}$ an $H$-Galois object. Then 
	\[
	\mathcal{P}=\sum_{\mathfrak{s} \in S_3} \sgn(\mathfrak{s})E_{\mathfrak{s}(1)}X_{\mathfrak{s}(2)}Y^d_{\mathfrak{s}(3)}
	\]
	is a polynomial $H$-identity for $A_{\sigma, a}{(\mathbb{G})}$. 
\end{proposition}
\begin{proof}
	Using the notation above, it is enough to calculate $\mu_{\alpha}(\mathcal{P})$ observing that
\begin{align*}
\mu_\alpha(E)&=t_1v_1,\\
\mu_\alpha(X)&=t_gv_g,\\
\mu_\alpha(Y)&=t_1v_y+t_yv_g,
\end{align*}
and, since $(t_1v_y)(t_yv_g)=q(t_yv_g)(t_1v_y)$, by Lemma \ref{qcommute},
\[
\mu_\alpha(Y^d)=t_1^d v_y^d+t_y^d v_g^d.
\]
Now, $\mathcal{P}=EXY^d-EY^dX+XY^dE-XEY^d+Y^dEX-Y^dXE$,
\begin{align*}
 \mu_{\alpha}(EXY^d) &= t_1v_1t_gv_g(t_1^dv_y^d+t_y^dv_g^d) \\
  &= t_1^{d+1} t_g v_g v_y^d + t_1 t_g t_y^d v_g^{d+1},
\end{align*}
and the computation of $\mu_{\alpha}$ on the remaining terms of $\mathcal{P}$ is easy and left to the reader. After a few
cancellations,
	\begin{align*}
	\mu_{\alpha}(\mathcal{P})&=t_1^{d+1}t_g(v_gv_y^d-v_y^dv_g) = (1-q^d)t_1^{d+1}t_gv_gv_y^d = 0
	\end{align*}
\end{proof}

The next theorem, a particular case of a result stated only for the type I monomial Hopf algebras in \cite[Theorem 3.4]{kassel}, remains valid for types II-VI with virtually the same proof.

\begin{theorem}\label{lema-kassel}
	Let $A_{\sigma,a}(\mathbb{G})$ and $A_{\tau,a}(\mathbb{G})$ be $H$-Galois objects for some $a\in k$ and $\sigma, \tau \in Z^2(G,k^{\times})$. Then
	\[
	I_H(A_{\sigma,a}(\mathbb{G}))=I_H(A_{\tau,a}(\mathbb{G}))
	\]
	if, and only if, $A_{\sigma,a}(\mathbb{G})$ and $A_{\tau,a}(\mathbb{G})$ are isomorphic as $H$-comodule algebras.
\end{theorem}
\begin{proof}
	Denote by $k^{\sigma}G$ the subalgebra of $A_{\sigma, a}(\mathbb{G})$ generated the symbols $u_x \ (x \in G)$ and relations $u_xu_{x'}=\sigma (x,x')u_{xx'}$ and $u_1=1$, for all $x,x' \in G$. If we prove that $I_{kG}(k^{\sigma}G)=I_{kG}(k^{\tau}G)$ then by \cite[Proposition 2.11]{aljadeffhaile} the cocycles $\sigma$ and $\tau$ are cohomologous. Therefore, by Proposition \ref{iso condition} $A_{\sigma,a}(\mathbb{G})$ and $A_{\tau,a}(\mathbb{G})$ are isomorphic as $H$-comodule algebras.
	It remains to prove that $I_{kG}(k^{\sigma}G)=I_{kG}(k^{\tau}G)$.
	
	Consider the diagram:
	\begin{equation*}
	\begin{tikzcd}[column sep=large, row sep=large]
	0\ar{r} & I_{kG}(k^{\sigma}G) \ar{r} \ar{d}[swap]{\imath}	&T(X_{kG})\ar{r}{\mu_{\alpha}}\ar{d}{\imath_T}&S(t_{kG})\otimes k^{\sigma}G \ar{d}{\imath_S}	\\
	0\ar{r} & I_H(A_{\sigma,a}(\mathbb{G})) \ar{r}	&T(X_{H})\ar{r}{\mu_{\alpha}}&S(t_{H})\otimes A_{\sigma,a}(\mathbb{G})
	\end{tikzcd}
	\end{equation*}
	
	The vertical map $\imath_T$ is induced by the inclusion $kG \longrightarrow H$.  It is injective. The map $\imath_S$ is induced by the previous natural inclusion and the comodule algebra inclusion  $k^{\sigma}G \subseteq A_{\sigma,a}(\mathbb{G})$. It sends its generators $t_i^xu_{x'}$ of $S(t_{kG})\otimes k^{\sigma}G$ to itself viewed as an element of $S(t_{H})\otimes A_{\sigma,a}(\mathbb{G})$. Note that the horizontal sequences are exacts by Theorem \ref{mu}. It is straightforward to check that the diagram is commutative. Therefore the restriction  $\imath$ of $\imath_T$ to $I_{kG}(k^{\sigma}G)$ send the latter to $I_H(A_{\sigma,a}(\mathbb{G}))$ and is injective. By this injectivity we have
	\[
	I_{kG}(k^{\sigma}G)=T(X_{kG}) \cap I_H(A_{\sigma,a}(\mathbb{G})).
	\]
	
	Then, as $	I_H(A_{\sigma,a}(\mathbb{G}))=I_H(A_{\tau,a}(\mathbb{G}))$, we have that $I_{kG}(k^{\sigma}G)=I_{kG}(k^{\tau}G)$. This completes the proof.
\end{proof}

For types III, V and VI we have shown that only isomorphism classes $[A_{\sigma, 0}(\mathbb{G})]$ and $[A_{\tau, 1}(\mathbb{G})]$ can occur (Proposition \ref{prop:iso3}). Next, we show that these classes are disjoint. First we need a Lemma.

\begin{lemma}\label{prop:H-poly}
Let $A_{\sigma, a}(\mathbb{G})$ be an $H$-Galois object. Then 
\[
\mathcal{Q}=(YX-qXY)^d-(1-q)^dX^dY^d
\] is a polynomial $H$-identity for $A_{\sigma, a}(\mathbb{G})$ if, and only if, $a=0$.
\end{lemma}
\begin{proof}
By Theorem \ref{mu}, it is enough to check that $\mu_{\alpha}(\mathcal{Q})=0$ if, and only if, $a=0$. Proceeding as in the proof of Proposition \ref{prop:H-PI}, we have
\begin{align*}
\mu_\alpha(YX-qXY)&=(1-q)t_yt_gv_g^2,\\
\mu_\alpha(E^d)&=t_1^du_1,\\
\mu_\alpha(X^d)&=t_g^dv_g^d,
\end{align*}
and
\begin{align*}
\mu_\alpha((YX-qXY)^d)&=(1-q)^dt_y^dt_g^dv_g^{2d}.
\end{align*}
Therefore,
\begin{align*}
\mu_{\alpha}(\mathcal{Q})&=(1-q)^dt_g^dvt_y^dv_g^{2d}-(1-q)^dt_g^dv_g^d(t_1^d v_y^d+t_y^d v_g^d)\\
&=-a(1-q)^dt_1^dt_g^du_g^du_{g^d}.
\end{align*}
Since $(1-q)^dt_1^dt_g^du_g^du_{g^d} \neq 0$, the result follows.
\end{proof}


\begin{proposition} \label{prop:polynomial}
	Let $\mathbb{G}$ be a group datum of type III, V or VI,  $A_{\sigma,a}(\mathbb{G})$ and $A_{\tau,b}(\mathbb{G})$ be $H$-Galois objects.
If $A_{\sigma,a}(\mathbb{G})\cong A_{\tau,b}(\mathbb{G})$ then
$a=b$.
\end{proposition}
\begin{proof}
Suppose that $a \neq b$. By Proposition \ref{prop:iso3}, we can assume that $a=0$ and $b=1$. Since $A_{\sigma,a}(\mathbb{G})\cong A_{\tau,b}(\mathbb{G})$ then $I_H(A_{\sigma,a}(\mathbb{G}))= I_H(A_{\tau,b}(\mathbb{G}))$ hence, it follows from Lemma \ref{prop:H-poly} that $\mathcal{Q}$ is a polynomial $H$-identity for $A_{\sigma,0}(\mathbb{G})$ but it is not a polynomial $H$-identity for $A_{\tau,1}(\mathbb{G})$, which is absurd.
\end{proof}

With this setup, the proof of Theorem \ref{thm:main} is straightforward.
\begin{proof}
The proof for type I was given by Kassel \cite[Theorem 3.4]{kassel}, so we may assume that $H$ is of types II through VI.

It follows from Proposition \ref{prop:iso3} that $A\cong A_{\sigma, a}(\mathbb{G})$ and $B\cong A_{\tau, b}(\mathbb{G})$ for some 2-cocycles $\sigma,\tau$ and $a,b\in k$.

For types II and IV, $a=b=0$ due to Proposition \ref{prop:iso2}, and the result follows at once from Theorem \ref{lema-kassel}.

For types III, V and VI, $a,b\in\{0,1\}$ thanks to Proposition \ref{prop:iso3}. Since $I_H(A_{\sigma,a}(\mathbb{G}))= I_H(A_{\tau,b}(\mathbb{G}))$, as in the proof of Proposition \ref{prop:polynomial}, we get $a=b$. Once again, the result follows from Theorem \ref{lema-kassel}.
\end{proof}

\bibliographystyle{plain}
\addcontentsline{toc}{chapter}{Bibliografia}
\bibliography{biblio}

\begin{thebibliography}{10}

\bibitem{aljadeffhaile}
Eli Aljadeff and Darrell Haile.
\newblock Simple {$G$}-graded algebras and their polynomial identities.
\newblock {\em Trans. Amer. Math. Soc.}, 366(4):1749--1771, 2014.

\bibitem{felipe}
Yuri Bahturin and Felipe Yasumura.
\newblock Graded polynomial identities as identities of universal algebras.
\newblock {\em Linear Algebra and its Applications}, 562:1 -- 14, 2019.

\bibitem{bichon}
Julien Bichon.
\newblock Galois and bigalois objects over monomial non-semisimple {H}opf
  algebras.
\newblock {\em J. Algebra Appl.}, 5(5):653--680, 2006.

\bibitem{chen}
Xiao-Wu Chen, Hua-Lin Huang, Yu~Ye, and Pu~Zhang.
\newblock Monomial {H}opf algebras.
\newblock {\em J. Algebra}, 275(1):212--232, 2004.

\bibitem{drensky}
V.~S. Drensky and M.~L. Racine.
\newblock Distinguishing simple {J}ordan algebras by means of polynomial
  identities.
\newblock {\em Comm. Algebra}, 20(2):309--327, 1992.

\bibitem{dascalescu}
Sorin D\u{a}sc\u{a}lescu, Constantin N\u{a}st\u{a}sescu, and \c{S}erban Raianu.
\newblock {\em Hopf algebras}, volume 235 of {\em Monographs and Textbooks in
  Pure and Applied Mathematics}.
\newblock Marcel Dekker, Inc., New York, 2001.
\newblock An introduction.

\bibitem{kasselbook}
Christian Kassel.
\newblock {\em Quantum groups}, volume 155 of {\em Graduate Texts in
  Mathematics}.
\newblock Springer-Verlag, New York, 1995.

\bibitem{kassel}
Christian Kassel.
\newblock Examples of polynomial identities distinguishing the {G}alois objects
  over finite-dimensional {H}opf algebras.
\newblock {\em Ann. Math. Blaise Pascal}, 20(2):175--191, 2013.

\bibitem{koshlukov}
Plamen Koshlukov and Mikhail Zaicev.
\newblock Identities and isomorphisms of graded simple algebras.
\newblock {\em Linear Algebra Appl.}, 432(12):3141--3148, 2010.

\bibitem{kreimer}
H.~F. Kreimer and P.~M. Cook, II.
\newblock Galois theories and normal bases.
\newblock {\em J. Algebra}, 43(1):115--121, 1976.

\bibitem{kush}
A.~Kh. Kushkule\u{\i} and Yu.~P. Razmyslov.
\newblock Varieties generated by irreducible representations of {L}ie algebras.
\newblock {\em Vestnik Moskov. Univ. Ser. I Mat. Mekh.}, (5):4--7, 1983.

\bibitem{montgomery}
Susan Montgomery.
\newblock {\em Hopf algebras and their actions on rings}, volume~82 of {\em
  CBMS Regional Conference Series in Mathematics}.
\newblock Published for the Conference Board of the Mathematical Sciences,
  Washington, DC; by the American Mathematical Society, Providence, RI, 1993.

\bibitem{radford}
David~E. Radford.
\newblock {\em Hopf algebras}, volume~49 of {\em Series on Knots and
  Everything}.
\newblock World Scientific Publishing Co. Pte. Ltd., Hackensack, NJ, 2012.

\bibitem{shestakov}
Ivan Shestakov and Mikhail Zaicev.
\newblock Polynomial identities of finite dimensional simple algebras.
\newblock {\em Comm. Algebra}, 39(3):929--932, 2011.

\end{thebibliography}
\clearpage \thispagestyle{empty}
\end{document}